\documentclass{amsart}
\usepackage{amsfonts,amssymb,amscd,amsmath,enumerate,verbatim,calc}
\usepackage[all]{xy}

\newcommand{\CM}{Cohen-Macaulay}

\newcommand{\wrt}{with respect to}

\newcommand{\n}{\mathfrak{n} }
\newcommand{\m}{\mathfrak{m} }

\newcommand{\rt}{\rightarrow}

\newcommand{\ov}{\overline}

\newcommand{\Ass}{\operatorname{Ass}}
\newcommand{\grade}{\operatorname{grade}}

\newcommand{\Spec}{\operatorname{Spec}}

\newcommand{\height}{\operatorname{height}}
\newcommand{\charp}{\operatorname{char}}

\newcommand{\injdim}{\operatorname{injdim}}

\newcommand{\Hom}{\operatorname{Hom}}
\newcommand{\Ext}{\operatorname{Ext}}

\newcommand{\Supp}{\operatorname{Supp}}

\theoremstyle{plain}

\newtheorem{theorem}{Theorem}[section]

\newtheorem{proposition}[theorem]{Proposition}

\theoremstyle{definition}
\newtheorem{definition}[theorem]{Definition}
\newtheorem{s}[theorem]{}
\newtheorem{remark}[theorem]{Remark}

\theoremstyle{remark}

\begin{document}

\title[Peskine-Szpiro Ideals]{On a generalization of a result of Peskine and Szpiro }
 \author{Tony J. Puthenpurakal}
\date{\today}
\address{Department of Mathematics, Indian Institute of Technology Bombay, Powai, Mumbai 400 076, India}
\email{tputhen@math.iitb.ac.in}
\begin{abstract}

Let $(R,\m)$ be a regular local ring containing a field  $K$. 
Let $I$ be a \CM \ ideal of height $g$. 
If $\charp K = p > 0$ then by a result of Peskine and Szpiro the local cohomology modules $H^i_I(R)$ vanish for $i > g$.
This result is not true if $\charp K = 0$.
However we prove that the Bass numbers of the local cohomology module $H^g_I(R)$ completely determine whether $H^i_I(R)$ vanish for $i > g$.
\end{abstract}

\maketitle
\textit{The result of this paper has been proved more generally for Gorenstein local rings by Hellus and Schenzel  \cite[Theorem 3.2]{HS}. However our result for regular rings is elementary to prove. In particular we do not use spectral sequences in our proof.}
\section{introduction}
Let $(R,\m)$ be a regular local ring containing a field $K$.  Motivated by a result of Peskine and Szpiro we make the following:

\begin{definition}
An ideal $I$  of $R$ is said to be a  \emph{Peskine-Szpiro ideal}  of $R$ if
\begin{enumerate}
\item
$I$ is a \CM \ ideal.
\item
$H^i_I(R) = 0 $ for all $i \neq \height I$.

\end{enumerate}

\end{definition}
  
Note that as $\height I = \grade I$ we have $H^i_I(R) = 0$ for $i < \height I$. Thus the only real condition for a \CM \ ideal  $I$  to be a Peskine-Szpiro ideal is that $H^i_I(R) = 0$ for $i > \height  I$.
In their fundamental paper  \cite[Proposition III.4.1]{PS} Peskine and Szpiro proved that if $\charp K = p > 0 $ then for all \CM \ ideals $I$ the local cohomology modules $H^i_I(R)$ vanish for $i > \height I$. This result is not true if $\charp K = 0$, for instance  see \cite[Example 21.31]{a7}.   We prove the following surprising result:
\begin{theorem}\label{main}
Let $(R,\m)$ be a regular local ring  of dimension $d$ containing a field $K$. Let $I$  be a \CM \ ideal of height $g$. The following conditions are equivalent:
\begin{enumerate}[\rm (i)]
\item 
$I$ is a Peskine-Szpiro ideal of $R$.
\item
For any prime ideal $P$ of $R$ containing $I$,  the Bass number
\[
\mu_i(P, H^g_I(R)) =  \begin{cases}  1 &\text{if }  \  i = \height P - g, \\
                                                                  0 & \text{otherwise.}    \end{cases}
\]
\end{enumerate}
\end{theorem}
Here the $j^{th}$ Bass number of an $R$-module $M$ with respect to a prime ideal $P$ is defined as $\mu_j(P,M) = \dim_{k(P)} \Ext^j_{R_P}(k(P), M_P)$ where $k(P)$ is the residue field of $R_P$.  Our result is essentially only an observation.

\s \label{HSL} We need the following remarkable properties of local cohomology modules over regular local rings containing a field (proved by Huneke and Sharp \cite{HuSh} if $\charp K = p > 0$ and by Lyubeznik  \cite{Lyu-1} if $\charp K = 0$).
Let  $(R,\m)$ be a regular ring containing a field $K$  and $I$ is an ideal in $R$.  Then the local cohomology modules of $R$ \wrt \ $I$ have the following properties:
\begin{enumerate}[\rm(i)]
\item
$H^j_{\m}(H^i_I(R))$ is injective.
\item
$\injdim_R H^i_I(R) \leq \dim \Supp H^i_I(R)$.
\item
The set of associated primes of $H^i_I(R)$ is finite.
\item
All the Bass numbers of $H^i_I(R)$ are finite.
\end{enumerate}
Here $\injdim_R H^i_I(R)$ denotes the injective dimension of $H^i_I(R)$. Also $\Supp M = \{ P \mid  M_P \neq 0 \ \text{and $P$ is a prime in $R$}\}$ is the support of an $R$-module $M$. 

\section{Permanence properties of Peskine-Szpiro ideals}
In this section we prove some permanence properties of Peskine-Szpiro ideals. We also show that if $\dim R - \height I \leq 2$ then a \CM \ ideal $I$ is a Peskine-Szpiro ideal.

\begin{proposition}\label{perm}
Let $(R,\m)$ be a regular local ring containing a field $K$. Let $I$ be a Peskine-Szpiro ideal of $R$.  Let $g = \height I$.
\begin{enumerate}[\rm (1)]
\item
Let $P$ be a prime ideal in $R$ containing $I$. Then $I_P$ is a Peskine-Szpiro ideal of $R_P$.
\item
Assume $g <  \dim R$. Let $x \in \m \setminus \m^2$ be $R/I$-regular. Then the ideal $(I + (x))/(x)$ is a Peskine-Szpiro ideal of $R/(x)$. 
\end{enumerate}
\end{proposition}
\begin{proof}
(1)  Note $I_P$ is a \CM  \  ideal of height $g$ in $R_P$. Also  note that for $i \neq g$ we have
\[
H^i_{I_P}(R_P) = H^i_I(R)_P = 0.
\] 
Thus $I_P$ is a Peskine-Szpiro ideal of $R_P$.

(2) Note that $ J = (I + (x))/(x)$ is a \CM \ ideal of height $g$ in the regular ring $ \ov{R} = R/(x)$. The short exact sequence
\[
0 \rt R \xrightarrow{x} R \rt \ov{R} \rt 0,
\]
induces a long exact sequence  
\[
\cdots \rt H^i_I(R) \rt H^i_J(\ov{R}) \rt H^{i+1}_I(R) \rt \cdots.
\]
Thus $H^i_J(\ov{R}) = 0$ for $i > g$. Therefore  $J$ is a Peskine-Szpiro ideal of $\ov{R}$.
\end{proof}

We now show that \CM \ ideals of small dimensions are Peskine-Szpiro. This result is already known, However we give a proof due to lack of a suitable reference. 

\begin{proposition}\label{small}
Let $(R,\m)$ be a regular local ring containing a field $K$.  Let $I$ be a  \CM \ ideal with $\dim R - \height I  \leq  2$. Then $I$ is a Peskine-Szpiro ideal.
\end{proposition}
\begin{proof}
We  have nothing to show if $\charp K = p > 0$. So we assume $\charp K = 0$.
Let $\dim R = d$ and $\height I = g$.

If $g = d$ then $I$ is $\m$-primary. By Grothendieck vanishing theorem we have $H^i_I(R) = 0$ for $i > d$. So $I$ is a Peskine-Szpiro ideal.

Now consider the case when $g = d-1$. Note $\dim R/I  = 1$. So $\dim \widehat{R}/I \widehat{R} = 1$. By Hartshorne-Lichtenbaum theorem, cf. \cite[Theorem 14.1]{a7}, we have that $H^d_{I\widehat{R}}(\widehat{R}) = 0$. By  faithful flatness we get $H^d_I(R) = 0$.

Finally we consider the case when $g = d -2$. We choose a flat extension $(B,\n)$ of $R$ with $\m B = \n$, $B$ complete and $B/\n$ algebraically closed. We note that $B/IB$ is \CM \ and $\dim B/IB = 2$. As $B/IB$ is \CM \ we get that the punctured spectrum $\Spec^\circ(B/IB)$ is connected see \cite[Proposition 15.7]{a7}.  So $H^{d-1}_{IB}(B) = 0$ by a result due to Ogus \cite[2.11]{O}. By faithful faithfulness we get 
$H^{d-1}_I(R) = 0$. By an argument similar to the previous case we also get $H^d_I(R) = 0$.  

\end{proof}

\section{Proof of Theorem \ref{main}}
In this section we prove our main result. The following remarks are relevant.
\begin{remark}\label{rem}
\begin{enumerate}
\item
Notice for any ideal $J$ of height $g$ we have $\Ass H^g_J(R) = \{ P \mid P\supset J \ \text{and} \  \height P  = g \}$. Also for any such prime ideal $P$ we have $\mu_0(P, H^g_J(R))  = 1$.
\item
Let $I$ be a \CM \ ideal of height $g$ in a regular local ring. Let $P$ be an ideal of height $g+r$ and containing $I$.  We note that $\dim H^g_I(R)_P  = r$. So  by \ref{HSL}  we get  $\injdim_{R_P}H^g_I(R)_P \leq r$. Thus $\mu_i(P, H^g_I(R)) = 0$ for $i > r$.
\end{enumerate}
\end{remark}

Let us recall the following result due to Rees, cf. \cite[3.1.16]{BH}.
\s \label{rees} Let $S$ be a commutative ring and let $M$ and $N$ be $S$-modules. (We note that $S$ need not be Noetherian. Also $M, N$ need not be finitely generated as $S$-modules.)
Assume there exists $x \in S$ such that it is $S\oplus M$-regular and $x N = 0$. Set $T = S/(x)$. Then $\Hom_S(N, M) = 0$ and for $i \geq 1$ we have
\[
\Ext^i_S(N,M) \cong \Ext^{i-1}_T(N, M/xM).
\]

We now give:
\begin{proof}[Proof of Theorem \ref{main}]
We first  prove $ \rm{(i)} \implies \rm{(ii)}$.  So $I$ is a Peskine-Szpiro ideal.  We prove our result by induction on $d-g$. 

If $d-g = 0$ then $I$ is $\m$-primary. So $H^d_I(R) = H^d_\m(R)  = E_R(R/\m) $ the injective hull  of the residue field. Clearly $\mu_0(\m, H^d_I(R)) = 1$ and $\mu_i(\m , H^d_I(R)) = 0$ for $i \geq 1$. 

Now assume $d - g = 1$. If $P$ is a prime ideal of $R$ containing $I$ with $\height P = d-1$ then by \ref{rem} we have $\mu_0(P, H^{d-1}_I(R)) = 1$ and $\mu_i(P, H^{d-1}_I(R)) = 0$ for $i\geq 1$.  We now consider the case when $P = \m$. By \ref{rem} we have $\mu_0(\m, H^{d-1}_I(R)) = 0$. Choose $x \in \m \setminus \m^2$ which is $R/I$-regular.  Set $\ov{R} = R/(x)$, $\n = \m/(x)$ and $J = I\ov{R}  = (I + (x))/(x)$. Then $J$ is $\n$-primary. The exact sequence $0 \rt R \xrightarrow{x} R \rt \ov{R} \rt 0$ induces the following exact sequence in cohomology
\[
0 \rt H^{d-1}_I(R) \xrightarrow{x} H^{d-1}_I(R) \rt H^{d-1}_J(\ov{R}) \rt 0.
\]
Here we have used that $I$ is a Peskine-Szpiro ideal and $J$ is $\n$-primary. Thus by \ref{rees} we have for $i \geq 1$,
\[
\mu_i(\m, H^{d-1}_I(R)) = \mu_{i-1}(\n, H^{d-1}_J(\ov{R})) =  \begin{cases}  1 & \text{if} \ i = 1,\\  0 & \text{otherwise}.   \end{cases}
\]
Thus the result follows in this case.

Now consider the case when $d - g \geq 2$. Let $P$ be a prime ideal in $R$ containing $I$ of height $g + r$. We first consider the case when $P \neq \m$.  By \ref{perm} we get that $I_P$ is a Peskine-Szpiro ideal of height $g$ in $R_P$. Also $\dim R_P - g < d -g$. So by induction hypothesis we have
\[
\mu_i(P, H^i_I(R)) = \mu_i(PR_P, H^i_{IR_P}(R_P)) =  \begin{cases} 1 & \text{if} \ i = r, \\ 0 &\text{otherwise}. \end{cases}
\]
\end{proof}
We now consider the case when $P = \m$. By \ref{rem} we have $\mu_0(\m, H^{g}_I(R)) = 0$. Choose $x \in \m \setminus \m^2$ which is $R/I$-regular.  Set $\ov{R} = R/(x)$, $\n = \m/(x)$ and $J = I\ov{R}  = (I + (x))/(x)$. Then $J$ is  height $g$ Peskine-Szpiro ideal in $\ov{R}$, see \ref{perm}. The exact sequence $0 \rt R \xrightarrow{x} R \rt \ov{R} \rt 0$ induces the following exact sequence in cohomology
\[
0 \rt H^{g}_I(R) \xrightarrow{x} H^{g}_I(R) \rt H^{g}_J(\ov{R}) \rt 0.
\]
Here we have used that $I$ is a Peskine-Szpiro ideal in $R$ and $J$ is  Peskine-Szpiro ideal in $\ov{R}$. Thus by \ref{rees} we have for $i \geq 1$,
\[
\mu_i(\m, H^{g}_I(R)) = \mu_{i-1}(\n, H^{g}_J(\ov{R})) =  \begin{cases}  1 & \text{if} \ i-1 = d-1-g,\\  0 & \text{otherwise}.   \end{cases}
\]
For the latter equality we have used induction hypothesis on the Peskine-Szpiro ideal $J$ (as $\dim \ov{R} - \height J = d-1 - g $). We note that $i-1 = d-1-g$ is same as $i = d-g$. Thus we have
\[
\mu_i(\m, H^g_I(R)) = \begin{cases} 1 & \text{if} \ i = d -g, \\ 0 &\text{otherwise}. \end{cases} 
\]  

We now prove $\rm{(ii)} \implies \rm{(i)}$.  By Peskine and Szpiro's result we may assume $\charp K = 0$.  We prove the result by induction on $d - g$. If $d-g \leq 2 $ then the result holds by Proposition \ref{small}.  So we may assume $d - g \geq 3$. Let $P$ be a prime ideal in $R$  containing $I$ with $P \neq \m$.  The ideal $I_P$ is a \CM \ ideal of height $g$ in $R_P$ satisfying the condition  $\rm{(ii)}$ on Bass numbers of $H^g_{I_P}(R_P)$.  As $\dim R_P - g  < d -g$ we get by our induction hypothesis  that $I_P$ is a Peskine-Szpiro ideal in $R_P$. Thus $H^i_{I_P}(R_P) = 0$ for $i > g$. It follows that  $\Supp H^i_I(R) \subseteq \{ \m \}$ for $i > g$. Let $k = R/\m$ and let $E_R(k)$  be the injective hull of $k$ as a $R$-module. Then by \ref{HSL}  there exists non-negative integers $r_i$ with
\begin{equation}\label{ps-e-1}
H^i_I(R) = E_R(k)^{r_i} \quad \text{for} \  i > g.
\end{equation}
Choose $x \in \m \setminus \m^2$ which is $R/I$-regular.  Set $\ov{R} = R/(x)$, $\n = \m/(x)$ and $J = I\ov{R}  = (I + (x))/(x)$. Then $J$ is  height $g$ \CM \  ideal in $\ov{R}$.  The exact sequence $0 \rt R \xrightarrow{x} R \rt \ov{R} \rt 0$ induces the following exact sequence in cohomology
\begin{equation}\label{ps-e-2}
0 \rt H^{g}_I(R) \xrightarrow{x} H^{g}_I(R) \rt H^{g}_J(\ov{R}) \rt H^{g+1}_I(R) \xrightarrow{x} H^{g+1}_I(R) \rt \cdots
\end{equation}
We consider two cases: \\
\textit{Case 1 :}  $H^{g+1}_I(R) \neq 0$. \\
We note that $\Hom_R(\ov{R}, E_R(k)) = E_{\ov{R}}(k)$. Thus the short exact sequence 
$0 \rt R \xrightarrow{x} R \rt \ov{R} \rt 0 $ induces an exact sequence 
\begin{equation}\label{ps-e-3}
0 \rt E_{\ov{R}}(k) \rt E_R(k) \xrightarrow{x} E_R(k) \rt 0.
\end{equation}
By (\ref{ps-e-1}) and (\ref{ps-e-3})  the exact sequence (\ref{ps-e-2}) breaks down into two exact sequences
\begin{align}
\label{ps-e-4} 0 \rt H^g_I(R) &\xrightarrow{x} H^g_I(R) \rt V \rt 0, \\
\label{ps-e-5} 0 \rt V &\rt H^g_J(\ov{R}) \rt E_{\ov{R}}(k)^{r_{g+1}} \rt 0.
\end{align} 
As $J$ is a \CM \ ideal in $\ov{R}$ with $\dim \ov{R}/J  = d-1-g \geq 2$ we get by \ref{rem} that $\n \notin \Ass_{\ov{R}} H^g_J(\ov{R})$. It follows from (\ref{ps-e-5})  that $\mu_1(\n, V)  \geq r_{g+1} > 0$.
By (\ref{ps-e-4})  and \ref{rees} we get that 
\[
\mu_{2}(\m, H^g_I(R)) = \mu_1(\n, V) > 0.
\]
So by our hypothesis we get $d-g = 2$. This is a contradiction as we assumed $d-g \geq 3$.

\textit{Case 2 :} $H^{g+1}_I(R) = 0$. \\
By (\ref{ps-e-2}) we get a short exact sequence,
\[
0 \rt H^g_I(R) \xrightarrow{x} H^g_I(R) \rt H^g_J(\ov{R}) \rt 0.
\]
Again by \ref{rees} we get that the \CM \ ideal $J$ of $\ov{R}$ satisfies the conditions $\rm{(ii)}$ of our Theorem. As $\dim \ov{R} - \height J = d-g-1$ we get by induction hypothesis that 
$J$ is Peskine-Szpiro ideal in $\ov{R}$.  Thus $H^i_J(\ov{R}) = 0 $ for $i > g$. Using (\ref{ps-e-1}) and (\ref{ps-e-3}) it follows that $H^i_I(R) = 0$ for $i \geq g + 2$.
Also by our assumption $H^{g+1}_I(R) = 0$. Thus $I$ is a Peskine-Szpiro ideal of $R$.

\end{document}